\documentclass[a4paper, 11pt]{amsart}
\usepackage{amssymb, enumerate}
\numberwithin{equation}{section}
\newtheorem{thm}{Theorem}[section]
\newtheorem{cor}[thm]{Corollary}
\newtheorem{lem}[thm]{Lemma}
\newtheorem{prop}[thm]{Proposition}

\theoremstyle{definition}

\newtheorem{example}[thm]{Example}

\DeclareMathOperator{\re}{Re}

\begin{document}
\title[Circle embeddings with restrictions on Fourier coefficients]{Circle embeddings with restrictions on Fourier coefficients}

\author[L. Li]{Liulan Li}
\address{Liulan Li, College of Mathematics and Statistics,
Hengyang Normal University, Hengyang,  Hunan 421002, People's
Republic of China.} \email{lanlimail2012@sina.cn}

\author[L. V. Kovalev]{Leonid V. Kovalev}
\address{Leonid V. Kovalev,
Department of Mathematics, Syracuse University, 215 Carnegie,
Syracuse, NY 13244, USA.} \email{lvkovale@syr.edu}



\begin{abstract} This paper continues the investigation of the relation between the geometry of a circle embedding and the values of its Fourier coefficients. First, we answer a question of Kovalev and Yang concerning the support of the Fourier transform of a starlike embedding. An important special case of circle embeddings are homeomorphisms of the circle onto itself. Under a one-sided bound on the Fourier support, such homeomorphisms are rational functions related to Blaschke products. We study the structure of rational circle homeomorphisms and show that they form a connected set in the uniform topology.
\end{abstract}

\maketitle \pagestyle{myheadings} \markboth{L.
Li and L. V. Kovalev}{Circle embeddings with restrictions on Fourier coefficients}

\section{Introduction}

Every continuous map $f$ of the unit circle
$\mathbb{T}=\{z\in\mathbb{C}\colon |z|=1\}$ to the complex plane extends to a harmonic map $F$ of the unit disk $\mathbb D=\{z\in\mathbb{C}\colon |z|<1\}$, and the Taylor coefficients of $F$ are given by the Fourier coefficients of $f$, namely $\hat{f}(n)=\frac{1}{2\pi}\int^{2\pi}_{0} f(e^{i\theta})e^{-i n\theta}d\theta$. This simple but important relation has consequences both for geometric function theory and for the theory of minimal surfaces~\cite{Du2, Ha, He}, especially when $f$ is an embedding, i.e., an injective continuous map. For example, when $f$ is a sense-preserving embedding with a convex image, the Rad\'o-Kneser-Choquet theorem~\cite[p. 29]{Du2} states that $F$ is a sense-preserving diffeomorphism and in particular $\hat f(1)\ne 1$. On the other hand, for every integer $N$ there exists a sense-preserving embedding $f\colon \mathbb T\to \mathbb{C}$ such that $\hat f(n)=0$ whenever $|n|\le N$~\cite[Theorem 5.1]{KY}. Thus, some restrictions on the shape of $f(\mathbb T)$ are necessary to obtain a non-vanishing result for $\hat f$.

By a circle embedding we mean an injective continuous map $f\colon \mathbb T\to\mathbb C$. In the special case $f(\mathbb{T})=\mathbb{T}$ the map $f$ is called a circle homeomorphism. The curve $f(\mathbb{T})$ is called shar-shaped about $w_0\in \mathbb C$ if the argument of $f(e^{i\theta}) - w_0$ is a monotone function of $\theta$. In this case $f$ is called a starlike embedding. In this paper we answer Question~5.1 in~\cite{KY} by proving that for a starlike embedding $f$, at least one of the coefficients $\hat{f}(1)$ and $\hat{f}(-1)$ is nonzero.

If the Fourier series of a sense-preserving circle homeomorphism $f\colon \mathbb T\to \mathbb T$ terminates in either direction, then $f$ is the restriction of the quotient of two Blaschke products~\cite[Proposition 3.2]{KY}.  It is difficult to describe exactly which ratios of Blaschke products restrict to circle homeomorphisms. We give some sufficient conditions in Section~\ref{circle homeomorphism}. In Section~\ref{the space of circle homeomorphisms}, specifically Theorems~\ref{space of special circle homeomorphism} and~\ref{space of circle homeomorphism} we show that the set of rational circle homeomorphisms is connected in the uniform topology.

\section{Main results and preliminaries}\label{preliminaries}

Hall~\cite[Theorem 2]{Ha} proved that $|\hat{f}(1)|+|\hat{f}(0)|>0$ when $f$ is a sense-preserving embedding and $f(\mathbb T)$ is star-shaped about  $0$. The first of our main results, Theorem~\ref{starlike}, allows $f(\mathbb T)$ to be star-shaped about any point; it also applies to sense-reversing embeddings.

\begin{thm}\label{starlike} Let $f\colon \mathbb{T}\to \mathbb{C}$ be a starlike embedding. Then
\[
|\hat{f}(1)|+|\hat{f}(-1)|>0.
\]
\end{thm}
Under the assumptions of this theorem, each of the individual coefficients $\hat{f}(1)$ and $\hat{f}(-1)$ may vanish~\cite[Proposition 5.1]{KY}.

A complex-valued function is called harmonic if its real and imaginary parts are harmonic. For any continuous function on $\mathbb{T}$, convolution with the Poisson kernel
\begin{equation}\label{def-Poisson-kernel}
P(z,\zeta)  = \frac{1-|z|^2}{|\zeta-z|^2}, \quad z\in \mathbb D, \ \zeta\in \mathbb T
\end{equation}
provides a harmonic extension in $\mathbb{D}$ \cite[Theorem 4.22]{Ah}. This fact is restated as a proposition for future references.

\begin{prop}\label{harmonic extension}\cite[Section 19.1]{Co} If $f:\ \mathbb{T}\rightarrow \mathbb{C}$ is continuous, then the series
\[
F(z)=\sum^{\infty}_{n=0}\hat{f}(n)z^n+\sum^{\infty}_{n=1}\hat{f}(-n)\overline{z}^n
\]
defines a harmonic function in $\mathbb{D}$, for which $f$ provides a continuous boundary extension.
\end{prop}

Theorem~\ref{starlike} has an implication for the harmonic extension of starlike embeddings.

\begin{cor}\label{extension} Suppose $f\colon \mathbb{T}\to \mathbb{C}$ is a starlike embedding. Let $F$ be the harmonic extension of $f$, that is the harmonic map $F\colon \mathbb D \to \mathbb C$ that agrees with $f$ on the boundary of $\mathbb D$. Then $F$ has nonvanishing total derivative, meaning that $|F_z| + |F_{\bar z}| > 0$ in $\mathbb D$.
\end{cor}

In contrast to the Rad\'o-Kneser-Choquet theorem, the harmonic map in Corollary~\ref{extension} need not be a diffeomorphism, see Example~\ref{folding}.

The circle homeomorphisms whose Fourier series terminates in one direction can be completely described in terms of finite Blaschke products. A Blaschke product of degree $n$ is a rational function of the form
\[
B(z) = \sigma \prod_{k=1}^n \frac{z-z_k}{1-\overline{z_k}z},
\]
where $z_1,\dots, z_n\in \mathbb{D}$ and $\sigma\in \mathbb T$. The book~\cite{GMR} is a convenient reference on the properties of these products.

\begin{lem}\label{terminating}\cite[Proposition 3.2]{KY} Suppose that $f:\mathbb{T}\rightarrow\mathbb{T}$ is a sense-preserving circle homeomorphism.
The set $\{n\in\mathbb{Z}\colon \hat{f}(n) \neq0\}$ is:

(a) bounded below if and only if
$f(\zeta)=B(\zeta)/\zeta^{n-1}$ for some integer $n\geq 0$, where
$B$ is a Blaschke product of degree $n$;

(b) bounded above if and only if
$f(\zeta)=\zeta^{n+1}/B(\zeta)$ for some integer $n\geq 0$, where
$B$ is a Blaschke product of degree $n$.
\end{lem}

Not every Blaschke product $B$ induces a circle homeomorphism in the way described in Lemma~\ref{terminating}. To treat the two cases of Lemma~\ref{terminating} in a unified way, Kovalev and Yang gave the necessary and sufficient condition  for the quotient of two finite Blaschke
products to be a circle homeomorphism, which is recorded in the following.

\begin{lem}\label{homeomorphism condition}\cite[Lemma 3.1]{KY} Suppose that $B_1, B_2$ are finite
Blaschke products:
\[
B_1(z)=\sigma_1\prod_{k=1}^n\frac{z-z_k}{1-\overline{z_k}z}, \quad
B_2(z)=\sigma_2\prod_{k=1}^m\frac{z-w_k}{1-\overline{w_k}z},
\]where
$z_1, z_2,\dots, z_n, w_1, w_2,\dots, w_m\in\mathbb{D}$ and $\sigma_1,
\sigma_2\in\mathbb{T}$. Then the quotient $B_1(\zeta)/B_2(\zeta)$ is
a sense-preserving circle homeomorphism if and only if $n-m=1$ and
\begin{equation}\label{poisson kernel condition}
\sum^n_{k=1}P(z_k, \zeta)\geq\sum^m_{k=1}P(w_k,
\zeta)\;\ \mbox{for\;\ all}\;\ \zeta\in\mathbb{T},
\end{equation}
where $P$ is the Poisson kernel for the unit disk $\mathbb{D}$, see~\eqref{def-Poisson-kernel}.
\end{lem}

Because condition~\eqref{poisson kernel condition} is difficult to check in practice, in Section \ref{circle homeomorphism} we continue to consider the problem of describing the Blaschke products that induce circle homeomorphisms, and obtain some explicit sufficient conditions in terms of the Blaschke zeros. Moreover, in Section~\ref{the space of circle homeomorphisms} we prove that both sets of homeomorphisms described in Lemma~\ref{terminating} are connected, as is the larger set in Lemma~\ref{homeomorphism condition}. The connectedness is understood
in the topology of $C(\mathbb T)$ induced by the uniform norm.

Since replacing the function $f(e^{i\theta})$ with $f(e^{-i\theta})$, which reverses the orientation of $f$, only changes the indexing of its Fourier coefficients, throughout the rest of this paper we will assume that $f$ is sense-preserving.

\section{Fourier coefficients of starlike embeddings}\label{starlike embeddings}

\begin{proof}[Proof of Theorem~\ref{starlike}] Suppose that the curve $f(\mathbb T)$ is star-shaped about some point
$w_0$. Without loss of generality, we may assume that $w_0=0$. Write
$f(e^{i\theta})=R(\theta)e^{i\phi(\theta)}$, where $\phi$ is
non-decreasing and \[\phi(2\pi-)-\phi(0)=2\pi.\]

Since $f$ is an embedding, the function $\phi$ is continuous on $[0, 2\pi)$.
Therefore the function $\psi(\theta): = \phi(\theta+\pi)-\phi(\theta)-\pi$ is continuous on $[0, \pi)$. Since $\psi(0) + \psi(\pi-) = 0$, the intermediate value theorem implies that there exists $\theta_0\in [0, \pi)$ such that $\psi(\theta_0)=0$, that is
\[
\phi(\theta_0+\pi)=\phi(\theta_0)+\pi.
\]
Let $z_0=e^{i\theta_0}$ and $g(e^{i\theta})=e^{i\phi(\theta)}$. Then
we have $f(e^{i\theta})=R(\theta)g(e^{i\theta})$ and
\[
g(-z_0)=g(e^{i(\theta_0+\pi)})=e^{i\phi(\theta_0+\pi)}=e^{i(\phi(\theta_0)+\pi)}=-e^{i\phi(\theta_0)}=-g(e^{i\theta_0})=-g(z_0).
\]
Let $\eta=g(z_0)$ and $\gamma$ be the open half-circle traced counterclockwise
from the point $z_0$ to $-z_0$. Note that $\mathbb T\setminus \overline{\gamma}$ is a complementary open half-circle traced counterclockwise
from $-z_0$ to $z_0$. Then $\overline{\eta}g(e^{i\theta})$
maps $\gamma$ onto the upper half-circle and $\mathbb T\setminus \overline{\gamma}$ onto the lower half-circle, for $g$ is sense-preserving.

Let
\[h(e^{i\theta})=\frac{\overline{z_0}e^{i\theta}}{1-\overline{z_0}^2e^{2i\theta}}.\]
Then $h$ maps $\gamma$ onto the half-line starting from the point
$\frac{1}{2}i$ to infinity along the upper imaginary axis, and
$\mathbb T\setminus \overline{\gamma}$ onto the half-line starting from the point
$-\frac{1}{2}i$ to infinity along the lower imaginary axis. Then
$\frac{1}{h(e^{i\theta})}=(1-\overline{z_0}^2e^{2i\theta})e^{-i\theta}z_0$
maps $\gamma$ onto the segment $(0, -2i]$ and $\mathbb{T}\setminus \gamma$
onto the segment $(0, 2i]$. In addition, $1/h(\pm z_0) = 0$.

The above statements yield that for all $e^{i\theta}\in\mathbb{T}$,
we have
\[\re\left(\overline{\eta}g(e^{i\theta})\frac{1}{h(e^{i\theta})}\right)=
\re\left(z_0\overline{\eta}g(e^{i\theta})(1-\overline{z_0}^2e^{2i\theta})e^{-i\theta}\right)\geq0,\]
where the equality holds if and only if $e^{i\theta}\in\{z_0,
-z_0\}.$

We now consider
\[\overline{\eta}z_0\left(\hat{f}(1)-\overline{z_0}^2\hat{f}(-1)\right)=\frac{1}{2\pi}\int^{2\pi}_0 R(\theta)
\overline{\eta}g(e^{i\theta})\frac{1}{h(e^{i\theta})}d\theta, \]
which implies that
\[|\hat{f}(1)|+|\hat{f}(-1)|\geq
\re\left(\overline{\eta}z_0\left(\hat{f}(1)-\overline{z_0}^2\hat{f}(-1)\right)\right)=\frac{1}{2\pi}\int^{2\pi}_0
R(\theta)\re\left(\overline{\eta}g(e^{i\theta})\frac{1}{h(e^{i\theta})}\right)d\theta>0.\]
Therefore, the proof is completed.
\end{proof}

\begin{proof}[Proof of Corollary~\ref{extension}] Since $F$ is the harmonic extension of $f$, by Proposition~\ref{harmonic extension} we have $F_z(0)=\hat f(1)$ and $F_{\bar z}(0) = \hat f(-1)$. By Theorem~\ref{starlike}, $|F_z(0)|+|F_{\bar z}(0)| > 0$. Given any point $a\in \mathbb D$, consider the M\"obius transformation $g(z) = \frac{z+a}{1+\bar a z}$. The composition $F\circ g$ is a harmonic map  with boundary values $f\circ g$. By the above, the total derivative of $F\circ g$ at $0$ does not vanish. Since $g(0)=a$, the chain rule shows that $|F_z(a)| + |F_{\bar z}(a)| > 0$.
\end{proof}

\begin{example}\label{folding} The embedding $f\colon \mathbb T\to \mathbb C$, defined by
\[
f(e^{i\theta}) = e^{i\theta} + e^{2i\theta} + \frac{1}{2}e^{-2i\theta}
\]
is starlike with respect to the point $w=1$. Its harmonic extension $F(z) = z+z^2+\bar z^2/2$ has Jacobian of variable sign in $\mathbb D$: in particular, $F_z$ vanishes at $-1/2$ while $F_{\bar z}$ vanishes at $0$. Therefore, $F$ is not a diffeomorphism.

For the sake of completeness we show that $f$ is indeed starlike with respect to $1$. Its values on $\mathbb T$ agree with the rational function $h(z) = z^2 + z  +  z^{-2}/2$. Nevanlinna's criterion for starlikeness~\cite[Theorem 8.3.1]{Goo} requires $zh'(z)/(h(z)-1)$ to have nonnegative real part on $\mathbb T$. We have
\[
\re\frac{zh'(z)}{h(z)-1}
= \re \frac{2z^2+z-z^{-2}}{z^2 + z + z^{-2}/2 -1}.
\]
Multiplying the numerator of this fraction by the conjugate of its denominator and writing $z=e^{i\theta}$, we arrive at
\[
\re\left\{(2z^2+z-z^{-2})(z^{-2} + z^{-1} + z^{2}/2 -1)\right\}
= \frac{5}{2} + 2\cos \theta - \cos 2\theta - \frac{1}{2}\cos 3\theta.
\]
The latter expression factors as $(5-2\cos 2\theta)\cos^2(\theta/2)$ and is therefore nonnegative.
\end{example}

\section{The set of rational circle homeomorphisms is connected}\label{the space of circle homeomorphisms}

The space $C(\mathbb T)$ of continuous mappings from $\mathbb T$ into $\mathbb C$ is equipped with the topology induced by the norm $\|f\| = \sup_{\mathbb T} |f|$. Let $H_+(\mathbb T)\subset C(\mathbb T)$ denote the group of all sense-preserving circle homeomorphisms $f\colon \mathbb T\to\mathbb T$. The group $H_+(\mathbb T)$ contains the rotation group $SO(2, \mathbb R)$ which consists of the maps $z\mapsto \sigma z$, $\sigma\in \mathbb T$. Proposition~4.2 in~\cite{Ghys} shows that $SO(2, \mathbb R)$ is a deformation retract of $H_+(\mathbb T)$, meaning there exists a continuous map $F\colon H_+(\mathbb T)\times [0, 1]\to H_+(\mathbb T)$ such that $F(\cdot, 1)$ is the identity and $F(\cdot, 0)$ is a retraction from $H_+(\mathbb T)$ onto $SO(2, \mathbb R)$. In particular, $H_+(\mathbb T)$ is a connected set. We will show that similar results hold for certain subsets of $H_+(\mathbb T)$ that consist of rational functions.

\begin{thm}\label{space of special circle homeomorphism} For each $n\geq 1$, the sets
\[
H_{n 1}(\mathbb T) = H_+(\mathbb T)\cap \left\{B(\zeta)/\zeta^{n-1}   \colon \text{$B$ is a Blaschke product of degree $n$}\right\}
\]
and
\[
H_{n 2}(\mathbb T) = H_+(\mathbb T)\cap \left\{\zeta^{n+1}/B(\zeta)  \colon \text{$B$ is a Blaschke product of degree $n$}\right\}
\]
contain $SO(2, \mathbb R)$ as a deformation retract. In particular, both sets are  connected. The unions $\bigcup_{n=1}^\infty H_{n 1}(\mathbb T)$ and $\bigcup_{n=1}^\infty H_{n 2}(\mathbb T)$ are connected as well.
\end{thm}

The proof is based on the semigroup property of the Poisson kernel $P$, namely $P_{r\rho} = P_r * P_\rho$ for $r, \rho \in [0, 1)$.

\begin{lem} \label{t circle homeomorphism} Let
\[
B(t, \zeta)=\sigma\prod_{k=1}^n \frac{\zeta-tz_k}{1-\overline{tz_k}\zeta},
\]
where $z_1, z_2, \dots, z_n\in\mathbb{D}$, $0\leq t\leq 1$ and
$\sigma\in\mathbb{T}$.

(a) If $B(1, \zeta)/\zeta^{n-1}$ is a circle
homeomorphism, then for any $t\in [0,1)$, $B(t, \zeta)/\zeta^{n-1}$
is still a circle homeomorphism.

(b) If $\zeta^{n+1}/B(1, \zeta)$ is a circle
homeomorphism, then for any $t\in [0,1)$, $\zeta^{n+1}/B(t, \zeta)$
is still a circle homeomorphism.
\end{lem}

\begin{proof} In part (a), by Lemma \ref{homeomorphism condition} we have
\begin{equation}\label{Poisson-sum-1}
    \sum_{k=1}^n P(z_k, \zeta) \ge n-1, \quad \zeta\in \mathbb T
\end{equation}
and the goal is to show that for any $t\in [0,1)$,
\begin{equation}\label{Poisson-sum-2}
    \sum_{k=1}^n P(tz_k, \zeta) \ge n-1, \quad \zeta\in \mathbb T.
\end{equation}
By the semigroup property of $P$,
\begin{equation}\label{convolution-sum}
\begin{split}
\sum_{k=1}^n P(tz_k, \zeta)
&= \int_{\mathbb T} \sum_{k=1}^n P(z_k, \xi) P(t, \zeta/\xi) \frac{|d\xi|}{2\pi} \\
&\ge \int_{\mathbb T} (n-1) P(t, \zeta/\xi) \frac{|d\xi|}{2\pi} = n-1.
\end{split}
\end{equation}
The proof of part (b) follows the same process except the inequalities \eqref{Poisson-sum-1} and \eqref{Poisson-sum-2} are changed from $\cdots\ge n-1$ to $\cdots\le n+1$.
\end{proof}

\begin{proof}[Proof of Theorem~\ref{space of special circle homeomorphism}]
We introduce a homotopy $F\colon H_{n1}(\mathbb T)\times [0, 1]\to H_{n1}(\mathbb T)$ using the notation of Lemma~\ref{t circle homeomorphism}: if  $\psi(\zeta) = B(1, \zeta)/\zeta^{n-1}$, then
$F(\psi, t)$ is the map $\zeta \mapsto B(t, \zeta)/\zeta^{n-1}$. In view of Lemma~\ref{t circle homeomorphism} and the fact that $B(0, \zeta)/\zeta^{n-1}$ is a rotation, the map $F$ is indeed a deformation retraction onto $SO(2, \mathbb R)$. In particular, each set $H_{n1}(\mathbb T)$ is connected. Since all these sets contain the identity map, their union is connected as well. The same reasoning applies to $H_{n 2}(\mathbb T)$.
\end{proof}

We have a similar result for the larger set of all circle homeomorphisms that are restrictions of rational functions. Recall that if a rational function maps $\mathbb T$ into $\mathbb T$, it must be a quotient of Blaschke products~\cite[Corollary 3.5.4]{GMR}.

\begin{thm}\label{space of circle homeomorphism} The set of all circle
homeomorphisms of the form $B_1/B_2$, where $B_1$, $B_2$ are finite
Blaschke products, is connected.
\end{thm}

\begin{proof} We will prove that each circle homeomorphism of the form
$B_1/B_2$, where $B_1$, $B_2$ are finite
Blaschke products, can be connected to
a linear map $z\mapsto \sigma z$. By Lemma \ref{homeomorphism condition}, we may assume that
\[ B_1(z)=\sigma_1\prod_{k=1}^n\frac{z-z_k}{1-\overline{z_k}z},\quad
B_2(z)=\sigma_2\prod_{k=1}^{n-1}\frac{z-w_k}{1-\overline{w_k}z}, \]
where $z_1, z_2,\dots, z_n, w_1, w_2,\dots, w_{n-1}\in\mathbb{D}$ and
$\sigma_1, \sigma_2\in\mathbb{T}$, and we have
\begin{equation}\label{Poisson-sum-2b}
 \sum^n_{k=1}P(z_k, \zeta)\geq\sum^{n-1}_{k=1}P(w_k,
\zeta)\;\ \mbox{for\;\ all}\;\ \zeta\in\mathbb{T}.\end{equation}
It suffices to prove that for all $t\in[0,1)$,
\[
 \sum^n_{k=1}P(tz_k, \zeta)\geq\sum^{n-1}_{k=1}P(tw_k,
\zeta)\;\ \mbox{for\;\ all}\;\ \zeta\in\mathbb{T}. \]
By similar
reasoning as in the proof of Lemma \ref{t circle homeomorphism}, we
have
\[ \begin{split}
\sum^n_{k=1}P(tz_k, \zeta)-\sum^{n-1}_{k=1}P(tw_k,
\zeta) & = \int_{\mathbb T} \sum_{k=1}^n P(z_k, \xi) P(t, \zeta/\xi)
\frac{|d\xi|}{2\pi}-
\int_{\mathbb T} \sum_{k=1}^{n-1} P(w_k, \xi) P(t, \zeta/\xi) \frac{|d\xi|}{2\pi} \\
& = \int_{\mathbb T}  \left(\sum_{k=1}^n P(z_k, \xi)-\sum_{k=1}^{n-1}
P(w_k, \xi)\right) P(t, \zeta/\xi) \frac{|d\xi|}{2\pi}\geq 0
\end{split} \]
for all  $\zeta\in\mathbb{T}$. Therefore, the proof is completed.
\end{proof}

\section{Sufficient conditions for rational circle homeomorphisms}\label{circle homeomorphism}

In this section we give some sufficient conditions for the ratio of Blaschke products to be a circle homeomorphisms. In certain cases these conditions are also necessary.

\begin{thm}\label{circle homeomorphism of n} Suppose that
\[
B(\zeta)=\sigma\prod_{k=1}^n \frac{\zeta-z_k}{1-\overline{z_k}\zeta},
\]
where $z_1, z_2, \dots, z_n\in\mathbb{D}$ and $\sigma\in\mathbb{T}$. If
\begin{equation}\label{L61a}
\sum_{k=1}^n \frac{1-|z_k|}{1+|z_k|}\geq n-1,
\end{equation}
then $B(\zeta)/\zeta^{n-1}$ is
a circle homeomorphism. In particular, this holds if $|z_k|\leq\frac{1}{2n-1}$ for all $k$.

If all numbers $z_k$ have the same argument, the condition~\eqref{L61a} is also necessary for $B(\zeta)/\zeta^{n-1}$ to be a circle homeomorphism.
\end{thm}

\begin{proof} For $k=1, \dots, n,$ we have
\begin{equation}\label{lowerPoisson}
P(z_k, \zeta) = \frac{1-|z_k|^2}{|\zeta-z_k|^2}
\ge \frac{1-|z_k|}{1+|z_k|}.
\end{equation}
If~\eqref{L61a} holds, then ~\eqref{lowerPoisson} implies $\sum_{k=1}^n P(z_k, \zeta) \ge n-1 $. It follows from Lemma~\ref{homeomorphism condition} that $B(\zeta)/\zeta^{n-1}$ is a circle homeomorphism.

In the case when all numbers $z_k$ have the same argument, we can choose $\zeta\in \mathbb T$ such that $\overline{\zeta} z_k\le 0$ for every $k$. Then equality holds in~\eqref{lowerPoisson}, which shows that~\eqref{L61a} is necessary in order to have $\sum_{k=1}^n P(z_k, \zeta) \ge n-1 $.
\end{proof}

In general it seems difficult to describe precisely when $B(\zeta)/\zeta^{n-1}$ is a circle homeomorphism, but for Blaschke products of degree $2$ the picture is more complete.

\begin{thm}\label{circle homeomorphism of 2} Suppose that
\[
B(\zeta)=\sigma  \frac{\zeta-a}{1-\overline{a}\zeta}
\frac{\zeta-b}{1-\overline{b}\zeta},
\]
where $a, b\in\mathbb{D}\cap \mathbb R$ and $\sigma \in \mathbb T$. Then $B(\zeta)/\zeta$ is a circle homeomorphism if and only if one of the following conditions holds.
\begin{enumerate}[(a)]
    \item $ab\ge 0$ and $1 - |a| - |b| - 3ab \ge 0$;
    \item $ab\le 0$ and $1  + ab - a^2 - b^2 \ge 0$.
\end{enumerate}
\end{thm}

\begin{proof} Part (a) follows from Theorem~\ref{circle homeomorphism of n} after observing that
\[
\frac{1-|a|}{1+|a|} + \frac{1-|b|}{1+|b|} - 1
= \frac{1 -|a| -|b| - 3|ab|}{(1+|a|)(1+|b|)}.
\]

In part (b), our goal is to determine when
$P(a,\zeta) + P(b,\zeta) \ge 1$ for all $\zeta\in \mathbb T$. The quantity
\[
P(a,\zeta) + P(b,\zeta) - 1
= \re\left(\frac{\zeta + a}{\zeta - a}
+\frac{2b}{\zeta - b}\right)
= \re \frac{\zeta^2 + (a+b)\zeta - 3ab}{\zeta^2 - (a+b)\zeta + ab}
\]
is nonnegative if and only if the product
\begin{equation}\label{c2b}
\left(\zeta^2 + (a+b)\zeta - 3ab\right)
\left(\bar \zeta^2 - (a+b)\bar\zeta + ab\right)
\end{equation}
has nonnegative real part. Expand~\eqref{c2b} and rewrite it in terms of $x = \re\zeta$, using the identity $\re (\zeta^2) = 2x^2-1$. The result is the quadratic function
\[
\begin{split}
g(x) & = -2ab (2x^2-1) + 4ab(a+b)x + 1 - (a+b)^2 - 3a^2b^2 \\
& = -4ab x^2 + 4ab(a+b)x + 1 - a^2 - b^2 - 3a^2b^2
\end{split}
\]
Since $ab\le 0$, the minimum of $g$ is attained at $x = (a+b)/2 \in (-1, 1)$. Thus, the inequality $\min_{[-1, 1]}g\ge 0$ is equivalent to $g((a+ b)/2)\ge 0$. Since
\[g\left(\frac{a+b}{2}\right) =
(1-ab) (1  + ab - a^2 - b^2)
\]
and $1-ab > 0$, part (b) is proved.
\end{proof}

The second natural way to form a circle homeomorphism from a Blaschke product of degree $n$ is $\zeta\mapsto \zeta^{n+1}/B(\zeta)$. For this construction we have a statement that parallels Theorem~\ref{circle homeomorphism of n}.

\begin{thm}\label{circle homeomorphism second kind} Suppose that
\[
B(\zeta)=\sigma\prod_{k=1}^n \frac{\zeta-z_k}{1-\overline{z_k}\zeta},
\]
where $z_1, z_2, \dots, z_n\in\mathbb{D}$ and $\sigma\in\mathbb{T}$. If
\begin{equation}\label{L62a}
\sum_{k=1}^n \frac{1+|z_k|}{1-|z_k|}\leq n+1,
\end{equation}
then $\zeta^{n+1}/B(\zeta)$ is
a circle homeomorphism. In particular, this holds if $|z_k|\leq\frac{1}{2n+1}$ for all $k$.

If all numbers $z_k$ have the same argument, the condition~\eqref{L62a} is also necessary for $\zeta^{n+1}/B(\zeta)$ to be a circle homeomorphism.
\end{thm}

\begin{proof} For $k=1, \dots, n,$ we have
\begin{equation}\label{upperPoisson}
P(z_k, \zeta) = \frac{1-|z_k|^2}{|\zeta-z_k|^2}
\le \frac{1+|z_k|}{1-|z_k|}.
\end{equation}
If~\eqref{L62a} holds, then ~\eqref{upperPoisson} implies $\sum_{k=1}^n P(z_k, \zeta) \le n+1$, and the conclusion follows from Lemma~\ref{homeomorphism condition}.

When all numbers $z_k$ have the same argument, we can choose $\zeta\in \mathbb T$ such that $\overline{\zeta} z_k\ge 0$ for every $k$. Then equality holds in~\eqref{upperPoisson}, which shows that~\eqref{L62a} is necessary in order to have $\sum_{k=1}^n P(z_k, \zeta) \le n+1$.
\end{proof}

Condition~\eqref{L62a} turns out to be more restrictive than ~\eqref{L61a}. Indeed, the inequality between arithmetic and harmonic means~\cite[Ch. 2]{HLP} implies
\[
\frac{1}{n} \sum_{k=1}^n \frac{1-|z_k|}{1+|z_k|}
\ge \left(\frac{1}{n} \sum_{k=1}^n \frac{1+|z_k|}{1-|z_k|}\right)^{-1}.
\]
If~\eqref{L62a} holds, then
\[
\frac{1}{n} \sum_{k=1}^n \frac{1-|z_k|}{1+|z_k|} \ge
\frac{n}{n+1} > \frac{n-1}{n},
\]
which implies ~\eqref{L61a}.

\section*{Acknowledgments}

The work of Leonid Kovalev was supported by the National Science Foundation grant DMS-1764266. The work of Liulan Li was partly supported by the Science and Technology Plan Project of Hunan Province (No. 2016TP1020), and the Application-Oriented Characterized Disciplines, Double First-Class University Project of Hunan Province (Xiangjiaotong [2018]469).

The authors thank the anonymous referee for several  suggestions which improved the organization of this article.


\begin{thebibliography}{99}

\bibitem {Ah} L. V. Ahlfors, Complex analysis. New York, McGraw-Hill Book Co., 1978.

\bibitem{Co} J. B. Conway, Functions of one complex variable. II. Graduate Texts in Mathematics, 159. Springer-Verlag, New York, 1995.

\bibitem{Du2} P. L. Duren, Harmonic mappings in the plane. Cambridge Tracts in Mathematics, 156. Cambridge University Press, Cambridge, 2004.

\bibitem{GMR} S. R. Garcia, J. Mashreghi and W. T. Ross, Finite Blaschke products and their connections. Springer, Cham, 2018.

\bibitem{Ghys} \'{E}. Ghys, Groups acting on the circle.
\textit{Enseign. Math.} (2) 47 (2001), no. 3--4, 329--407.

\bibitem{Goo} A. W. Goodman, Univalent functions. Vol. I. Mariner Publishing Co., Inc., Tampa, FL, 1983.

\bibitem{Ha} R. R. Hall, On an inequality of E. Heinz, \textit{J. Analyse Math.}
\textbf{42} (1982/83), 185--198.

\bibitem{HLP} G. H. Hardy, J. E. Littlewood, G. Plya. Inequalities. 2d ed. Cambridge University Press, 1952.

\bibitem{He} E. Heinz, \"{U}ber die L\"{o}sungen der Minimalfl\"{a}chengleichung, \textit{Nachr. Akad. Wiss. G\"{o}ttingen. Math.Phys. Kl.}
(1952), 51--56.

\bibitem{KY} L. V. Kovalev and X. Yang,
Fourier series of circle embeddings, \textit{Comput. Methods Funct.
Theory} \textbf{19} (2019), 323--340.

\end{thebibliography}
\end{document}